\documentclass[a4paper,reqno]{amsart}

\usepackage{geometry}
\title{Topology and Diffeology via Metric-like Functions}
\author{Masaki Taho}                                                                               

\makeatletter
\@namedef{subjclassname@2020}{\textup{2020} Mathematics Subject Classification}
\makeatother


\address{GRADUATE SCHOOL OF MATHEMATICAL SCIENCES, THE
UNIVERSITY OF TOKYO, 3-8-1 KOMABA, MEGURO-KU, TOKYO, 153-8914,
JAPAN}
\email{taho@ms.u-tokyo.ac.jp}

\usepackage{amsmath} 
\usepackage{amssymb} 
\usepackage{amsfonts} 
\usepackage{mathrsfs} 
\usepackage{amsthm} 
\usepackage{bm} 
\usepackage[new]{old-arrows} 
\usepackage[dvipdfmx]{graphicx} 
\usepackage{wrapfig} 
\usepackage{color} 
\usepackage[labelsep=colon]{caption} 
\usepackage[labelformat=empty,subrefformat=parens]{subcaption} 

\usepackage{spalign}

\usepackage{multicol} 
\usepackage{units} 
\usepackage[all]{xy} 
\usepackage{tikz-cd} 
\usepackage[backref]{hyperref} 
\hypersetup{
    colorlinks=true,
    linkcolor= [rgb]{0.3,0.55,0.25},
    citecolor=[rgb]{0.164,0.39,0.77},
    urlcolor=[rgb]{0,0,1}
}

\usepackage{cleveref}

\usepackage{thmtools}
\usepackage{thm-restate}  

\newtheorem{thm}{Theorem}[section]
\newtheorem{prop}[thm]{Proposition}
\newtheorem{lemma}[thm]{Lemma}
\newtheorem{cor}[thm]{Corollary}

\newtheorem*{mthm*}{Main Theorem}
\newtheorem*{thm*}{Theorem}
\newtheorem*{prop*}{Proposition}

\theoremstyle{definition}
\newtheorem{defi}[thm]{Definition}
\newtheorem{ex}[thm]{Example}
\newtheorem{remark}[thm]{Remark}

\newcommand{\topsp}{{\mathrm {Top}}}                 
\newcommand{\mfd}{{\mathrm {Mfd}}}                   
\newcommand{\dflg}{{\mathrm {Dflg}}}                 
\newcommand{\encl}{{\mathrm {Eucl}}}                 

\newcommand{\conti}{C}

\newcommand{\plotdom}[1]{U_{#1}}                      

\newcommand{\smoothmap}[2]{C^{\infty}(#1,#2)}

\newcommand{\embmfd}[2]{\Psi_{#1}(\mathbb{R}^{#2})}
\newcommand{\embgraph}[1]{\Phi{(#1)}}
\newcommand{\nhddr}[3]{\mathcal{U}_{#1,#2}(#3)}

\newcommand{\nhdgalatius}[3]{\mathscr{U}_{#1,#2}(#3)}

\newcommand{\mappingsp}[2]{C^0(#1,#2)}
\newcommand{\cptopen}[2]{W(#1,#2)}
\newcommand{\openball}[2]{B_{#1}{(#2)}}
\newcommand{\closedball}[1]{{#1}D}
\newcommand{\diffeomfdemb}[3]{\mathcal{P}_{#1}(#2,#3)}
\newcommand{\grassmannian}[2]{\mathop{\mathrm{Gr}}_{#1}(\mathbb{R}^{#2})}
\newcommand{\domain}[2]{s(#1, #2)}
\newcommand{\powerset}[1]{\mathcal{P}{(#1)}}
\newcommand{\prototile}{\mathscr{P}}
\newcommand{\Int}[1]{\operatorname{Int}{#1}}
\newcommand{\orbit}[1]{O{(#1)}}
\newcommand{\tilingsp}[1]{\Omega_{#1}}
\newcommand{\diam}[1]{\mathop{\mathrm{diam}}{#1}}
\newcommand{\Span}{\mathop{\mathrm{Span}}}
\newcommand{\entourage}[2]{U_{#1,#2}}

\newcommand{\IZ}{\cite{MR3025051}}
\newcommand{\SH}{\cite{MR3884529}}
\newcommand{\souriau}{\cite{MR0607688}}

\newcommand{\galatius}{\cite{MR2653727}}
\newcommand{\metricemb}{\cite{MR3243393}}
\newcommand{\graphemb}{\cite{MR2784914}}
\newcommand{\mfdembdflg}{\cite{MR4821355}}
\newcommand{\tilinginvertlimit}{\cite{MR1631708}}
\newcommand{\tilingtextbook}{\cite{MR2446623}}
\newcommand{\tilingdflg}{\cite{MR4882913}}
\newcommand{\broubakione}{\cite{MR1726779}}
\newcommand{\broubakitwo}{\cite{MR1726872}}
\newcommand{\quasiuniform}{\cite{MR660063}}
\newcommand{\gmtw}{\cite{MR2506750}}
\newcommand{\metriccantero}{\cite{Cantero}}

\newcommand{\thref}[1]{Theorem~\ref{#1}}
\newcommand{\proref}[1]{Proposition~\ref{#1}}
\newcommand{\defref}[1]{Definition~\ref{#1}}
\newcommand{\exref}[1]{Example~\ref{#1}}
\newcommand{\lemref}[1]{Lemma~\ref{#1}}

\newcommand{\remref}[1]{Remark~\ref{#1}}

\begin{document}

\begin{abstract}

    This paper investigates spaces equipped with a family of metric-like functions satisfying certain axioms. 
    These functions provide a unified framework for defining topology, uniformity, and diffeology.
    The framework is based on a family of metric-like functions originally introduced for spaces of submanifolds. 
    We show that the topologies, uniformities, and diffeologies of these spaces can be systematically derived from the proposed axioms. 
    Furthermore, the framework covers examples such as spaces with compact-open topologies, tiling spaces, and spaces of graphs, 
    which have appeared in different contexts. 
    These results support the study of spaces with metric-like structures from both topological and diffeological perspectives. 
\end{abstract}

\maketitle










\section{Introduction} \label{section intro}

The space $\embmfd{d}{N}$ of all closed $d$-dimensional submanifolds in $\mathbb{R}^N$ has been studied from various perspectives. 
Among these, Galatius--Randal-Williams introduced in \galatius~a specific topology on $\embmfd{d}{N}$, based on ideas that first appeared implicitly in the study of the cobordism category by Galatius-Madsen-Tillmann-Weiss \gmtw~and Galatius \graphemb.
Later, in Schommer-Pries \mfdembdflg, it was shown that this topology can be reinterpreted as the $D$-topology of a certain diffeology on $\embmfd{d}{N}$, 
providing a new perspective for studying this space. 

Diffeological spaces, introduced in \souriau~by Souriau in the 1980s, 
extend the notion of smooth manifolds to a broader class of spaces, including singular spaces and infinite-dimensional objects. 
Any diffeology naturally induces a topology, the $D$-topology, which is compatible with the smooth structure and provides a way to extend smooth structures while respecting the underlying topology. 
This makes them particularly useful for studying spaces like $\embmfd{d}{N}$, where standard manifold structures are not directly applicable. 


Fundamental properties of this topology, including its metrizability, were first established by B\"okstedt-Madsen in \metricemb. 
In \metricemb, a family of functions $\{d_r\colon \embmfd{d}{N} \times \embmfd{d}{N} \to \mathbb{R}_{\geq 0}\cup\{\infty\}\}_{r>0}$, indexed by $r>0$, 
was introduced to analyze the topology of $\embmfd{d}{N}$. These functions satisfy properties which resemble the axioms of a metric:

\begin{prop}[\metricemb, Lemma 4.2] \label{proposition reference embmfd axioms}
    Let $W_i\in \embmfd{d}{N}$ be $d$-dimensional closed submanifolds of $\mathbb{R}^N$. Then the following properties hold:
    \begin{description}
        \item[Symmetry]\leavevmode\\
         For all $r>0$, $d_r(W_1,W_2)=d_r(W_2,W_1)$. 
        \item[Upper semi-continuity]\leavevmode
         \begin{itemize}
            \item If $r_1\leq r_2$, then $d_{r_1}(W_1,W_2)\leq d_{r_2}(W_1,W_2)$. 
            \item If $d_r(W_1,W_2)<\varepsilon$, there exists $\delta > 0$ such that $d_{r+\delta}(W_1,W_2)<\varepsilon$.
        \end{itemize}
        \item[Weak triangle inequality]\leavevmode\\
        Let $r_1, r_2, r_3 > 0$. If $d_{r_1+r_2}(W_1,W_2)<r_3$ and $d_{r_1+r_2+r_3}(W_2,W_3)<r_2$, then 
        \[
            d_{r_1}(W_1,W_3)\leq d_{r_1+r_2}(W_1,W_2)+d_{r_1+r_2+r_3}(W_2,W_3).
        \]
    \end{description}
\end{prop}

Furthermore, in \metricemb, it was shown that these functions define a topology on $\embmfd{d}{N}$, which coincides with the topology defined in \galatius. 

\begin{prop}[\metricemb, Lemma 4.3, Theorem 2.1] \label{proposition reference embmfd topology}
    Given $M\in\embmfd{d}{N}$, $r>0$, and $\varepsilon>0$, define the neighborhoods
    \[
        \nhddr{r}{\varepsilon}{M}=\{N\in \embmfd{d}{N}\mid d_r(M, N)<\varepsilon\}.
    \]
    Then, the collection of sets $\nhddr{r}{\varepsilon}{M}$ forms a basis for a topology on $\embmfd{d}{N}$, 
    which coincides with the topology introduced in \galatius. 
\end{prop}


Motivated by these developments, this paper introduces an abstract framework for spaces equipped with families of metric-like functions $\{d_r\}_{r>0}$, 
satisfying axioms similar to those in \metricemb. We show that for such spaces, topology, uniformity, and diffeology can be defined in a unified manner. 
This approach not only generalizes the structures found on $\embmfd{d}{N}$ but also provides new insights into other spaces.

One example we focus on in this paper is the tiling space, which is well studied and described in \tilingtextbook~by Sadun as a metric space representing a space of tilings of $\mathbb{R}^d$.
In \tilingdflg~by Alatorre--Rodr{\'\i}guez-Guzm{\'a}n, a diffeology was defined on tiling spaces, and its basic properties were studied. 
We show that the topology on tiling spaces, as defined in \tilingtextbook, can naturally be reinterpreted as arising from a family of $\{d_r\}_{r>0}$. 
Moreover, we show that the diffeology defined using this family $\{d_r\}_{r>0}$ coincides with the one introduced in \tilingdflg. These results form the main theorem of this paper.

Additionally, we show that the axioms for $\{d_r\}_{r>0}$ introduced in this paper appear in various other contexts, 
including the compact-open topology and the space of all graphs in $\mathbb{R}^N$. 
In the case of the space of graphs, as studied by Galatius in \graphemb, a slight relaxation of the axioms—specifically, a further weakening of the triangle inequality—allows us to incorporate these examples naturally. 
Thus, this paper provides a unified framework for studying spaces defined by families of metric-like functions, offering a new perspective on their topology, uniformity, and diffeology. 
This work represents a first step toward a broader understanding of such spaces, with potential applications beyond the specific cases considered here. 

We also show that a family $\{d_r\}_{r>0}$ gives rise to a natural pseudo-metric that is not necessarily symmetric on the space. 
In particular, in the case of submanifolds, this provides a simpler way to recover the result of \metricemb, Theorem 4.4. 
Cantero \metriccantero~also provided an explicit metric that induces the same topology on submanifold spaces.
Our approach takes a different angle by using an axiomatic formulation, which can be naturally applied to other types of spaces as well.

\section{Topology defined by \texorpdfstring{$d_r$}{dr}}

We first introduce the following axioms, inspired by the properties stated in \proref{proposition reference embmfd axioms}, 
which were originally established in \metricemb. 

\begin{defi} \label{definition axioms of dr}
    Let $X$ be a set, and $\{d_r \colon X \times X \to \mathbb{R}_{\geq 0}\cup\{\infty\} \}_{r>0}$ be a family of functions indexed by $r>0$. 
    We assume that the following axioms hold for any three elements $x$, $y$, $z\in X$: 
    \begin{description}
        \item[Self-distance Axiom]\leavevmode\\
        For all $r>0$, $d_r(x, x)=0$. 
        \item[Upper semi-continuity]\leavevmode
            \begin{itemize}
                  \item If $r_1\leq r_2$, then $d_{r_1}(x, y)\leq d_{r_2}(x, y)$. 
                  \item If $d_r(x, y)<\varepsilon$, there exists $\delta > 0$ such that $d_{r+\delta}(x, y)<\varepsilon$.
            \end{itemize}
        \item[Weak triangle inequality]\leavevmode\\
        Let $r_1, r_2, r_3 > 0$. If $d_{r_1+r_2}(x,y)<r_3$ and $d_{r_1+r_2+r_3}(y, z)<r_2$, then 
        \[
            d_{r_1}(x, z)\leq d_{r_1+r_2}(x, y)+d_{r_1+r_2+r_3}(y, z).
        \]
    \end{description}
\end{defi}

\begin{remark}{}{}
  For simplicity, we use the notation $r_{i_1\dots i_k}$ to denote $r_{i_1}+\dots+r_{i_k}$. 
  Under this notation, the weak triangle inequality can be rewritten as follows:
  Let $r_1>0, r_2>0$, $r_3>0$. 
  If $d_{r_{12}}(x, y)<r_3$ and $d_{r_{123}}(y, z)<r_2$, then 
  $d_{r_1}(x, z)\leq d_{r_{12}}(x, y)+d_{r_{123}}(y, z)$.
\end{remark}

\begin{defi} \label{definition neighborhoods of topology}
  Let $(X, \{d_r\}_{r>0})$ satisfy the axioms in \defref{definition axioms of dr}.
  For a given $x\in X$, $r>0$, $\varepsilon>0$, we define the neighborhood
  \[
      \nhddr{r}{\varepsilon}{x}=\{y\in X\mid d_r(x, y)<\varepsilon\}.
  \]
\end{defi}

It is important to note that $d_r$ is not necessarily symmetric. Consequently, the order of the arguments in this definition must be carefully considered.
Moreover, observe that $x\in\nhddr{r}{\varepsilon}{x}$ always holds. 

The following proposition and its proof are identical to those given in \metricemb.

\begin{prop}[\metricemb, Lemma 4.3] \label{proposition basis of topology}
    Let $(X, \{d_r\}_{r>0})$ satisfy the axioms in \defref{definition axioms of dr}. Suppose $x\in X$, and let $r>0$ and $\varepsilon>0$. 
    If $y\in \nhddr{r}{\varepsilon}{x}$, then there exist $r'>0$ and $\varepsilon'>0$ such that 
    \[
        \nhddr{r'}{\varepsilon'}{y}\subset \nhddr{r}{\varepsilon}{x}. 
    \]
    In particular, the collection of sets $\nhddr{r}{\varepsilon}{x}$ forms a basis for a topology on X. 
\end{prop}

\begin{proof}
By the axioms of $\{d_r\}_{r>0}$, there exist $\delta>0$ and $\varepsilon_1>0$ such that $d_{r+\delta}(x, y)<\varepsilon-\varepsilon_1$. 
With these values of $\delta$ and $\varepsilon_1$, if we set $\varepsilon'=\min\{\delta, \varepsilon_1\}$ and $r'=r+\delta+\varepsilon-\varepsilon_1$, then the inclusion above holds. 
To prove the inclusion, let $z\in\nhddr{r'}{\varepsilon'}{y}$ be any element. Since $d_{r+\delta}(x,y)<\varepsilon-\varepsilon_1$ and $d_{r'}(y, z)<\varepsilon'\leq \delta$, we have
\[
    d_r(x, z)\leq d_{r+\delta}(x,y)+d_{r'}(y, z)\leq \varepsilon-\varepsilon_1+\varepsilon'\leq \varepsilon
\]
by the weak triangle inequality. Thus, $z\in \nhddr{r}{\varepsilon}{x}$. 
\end{proof}

\begin{cor} \label{corollary upper semiconti}
Let $(X, \{d_r\}_{r>0})$ satisfy the axioms in \defref{definition axioms of dr}. 
For any $x\in X$ and $r>0$, the function
\[
    d_r(x, \cdot)\colon X\to \mathbb{R}_{\ge 0}\cup\{\infty\},\quad y\mapsto d_r(x, y)
\]
is upper semi-continuous.
\end{cor}

\begin{proof}
Let $b\in \mathbb{R}$, and let $y\in d_r(x, \cdot)^{-1}((-\infty, b))=\nhddr{r}{b}{x}$ be any element. 
By \proref{proposition basis of topology}, there exist $r'>0$ and $\varepsilon>0$ such that 
\[
    \nhddr{r'}{\varepsilon}{y}\subset \nhddr{r}{b}{x}. 
\]
This implies that $d_r(x, \cdot)\colon X \to \mathbb{R}_{\geq 0}\cup \{\infty\}$ is upper semi-continuous.
\end{proof}

We now present some examples of spaces defined by $\{d_r\}_{r>0}$. 
In the first example, the family $\{d_r\}_{r>0}$ satisfies not only the weak triangle inequality but also the standard triangle inequality. Thus, this example may not be particularly surprising. 
However, it illustrates that spaces defined by $\{d_r\}_{r>0}$ can be viewed as generalizations of the compact-open topology.

\begin{ex} \label{example mapping space}
Let $(T, d_T)$ be a proper metric space, i.e., a metric space where every closed ball is compact. Let $(S, d_S)$ be another metric space, 
and consider the set $X=\mappingsp{T}{S}$ of all continuous maps from $T$ to $S$. 
Fix a point $t_0\in T$, and for each $r>0$, define 
\[
  \closedball{r}=\{t\in T\mid d_T(t, t_0)\leq r\}
\]
to be the closed ball centered at $t_0$ for each $r>0$. 
Then, for each $f, g\in X$, define 
\[
    d_r(f, g)=\sup_{x\in \closedball{r}} d_S(f(x),g(x)).
\] 

It is straightforward to verify that $(X, d_r)$ satisfies the axioms of \defref{definition axioms of dr}. 
Moreover, the topology induced by $\{d_r\}_{r>0}$ coincides with the compact-open topology on $\mappingsp{T}{S}$.
\end{ex}

\begin{proof}
    
Let $K\subset T$ be an arbitrary compact subset, and let $U\subset S$ be an arbitrary open subset. Consider an element $f\in \cptopen{K}{U}=\{g\colon T\to S\mid g(K)\subset U\}$. 
Since $K$ is compact, we can choose $r>0$ such that $K\subset \closedball{r}$. 
For each $x\in K$, there exists $\varepsilon_x>0$ such that $\openball{\varepsilon_x}{f(x)}\subset U$. 
By compactness of $K$, we can choose a finite set $\{x_1,\dots, x_n\}\subset K$ such that
\[
    K\subset \bigcup_{i=1}^n f^{-1}(\openball{\varepsilon_{x_i}/2}{f(x_i)}).
\]
Define $\varepsilon=\min\{\varepsilon_{x_i}/2\mid i=1,\dots, n\}$. Then, $\nhddr{r}{\varepsilon}{f}\subset \cptopen{K}{U}$. 

To prove the inclusion, let $g\in \nhddr{r}{\varepsilon}{f}$ and $x\in K$ be arbitrary. 
Since $x$ belongs to one of the preimages $f^{-1}(\openball{\varepsilon_{x_i}/2}{f(x_i)})$, we have
\[
  d_S(f(x_i),f(x))<\varepsilon_{x_i}/2. 
\]
Moreover, since $g\in \nhddr{r}{\varepsilon}{f}$, it follows that 
\[
    d_S(f(x), g(x))<\varepsilon\leq \varepsilon_{x_i}/2. 
\]
Thus, we conclude 
\[
    d_S(g(x), f(x_i))<\varepsilon_{x_i},
\]
which implies, $g(x)\in \openball{\varepsilon_{x_i}}{f(x_i)}\subset U$. 
Therefore, $\nhddr{r}{\varepsilon}{f}\subset \cptopen{K}{U}$. 

For the reverse inclusion, 
let $r>0$ and $\varepsilon>0$ be arbitrary, and consider $f\in \mappingsp{T}{S}$. 
For each $x\in \closedball{r}$, there exists $r_x>0$ such that $\openball{2r_x}{x}\subset f^{-1}(\openball{\varepsilon/2}{f(x)})$. 
Since $\closedball{r}$ is compact (by the assumption that $T$ is proper), there exists a finite set $\{x_1,\dots, x_n\}\subset \closedball{r}$ such that 
\[
    \closedball{r}\subset \bigcup_{i=1}^n \openball{r_{x_i}}{x_i}.
\]
For each $i=1,\dots, n$, define $W_i=\cptopen{\overline{\openball{r_{x_i}}{x_i}}}{\openball{\varepsilon/2}{f(x_i)}}$. Then, 
\[
    f\in \bigcap_{i=1}^n W_i\subset \nhddr{r}{\varepsilon}{f}.
\]
To prove the inclusion, let $g\in \bigcap_{i=1}^n W_i$ and $x\in \closedball{r}$. 
Since $x$ is contained in some $\openball{r_{x_i}}{x_i}$, we have $g(x)\in \openball{\varepsilon/2}{f(x_i)}$.
Moreover, since $\openball{r_{x_i}}{x_i}\subset f^{-1}(\openball{\varepsilon/2}{f(x_i)})$, we obtain $f(x)\in \openball{\varepsilon/2}{f(x_i)}$.
Thus, 
\[
    d_S(f(x), g(x))<\varepsilon
\]
for each $x\in \closedball{r}$. 
It follows that 
\[
    d_r(f,g)=\sup_{x\in \closedball{r}}d_S(f(x),g(x))=\max_{x\in \closedball{r}}d_S(f(x),g(x))<\varepsilon, 
\]
which shows that $g\in \nhddr{r}{\varepsilon}{f}$.
\end{proof}


The following example plays a fundamental role in the development of the concept of $\{d_r\}_{r>0}$.
\begin{ex}[\metricemb] \label{example submanifold}
  For $N\geq 2$ and $d<N$, we define 
  \[
    \embmfd{d}{N}=\{W^d\subset \mathbb{R}^N\mid \mbox{$W^d$ is a closed submanifold of $\mathbb{R}^N$ such that $\partial W=\emptyset$.}\}
  \]
  For each $r>0$, we define the closed ball of radius $r$ in $\mathbb{R}^N$ by $\closedball{r}=\{x\in \mathbb{R}^N\mid |x|\leq r\}$. For each $r>0$ and $V,W\in \embmfd{d}{N}$, we define 
  \[
    \diffeomfdemb{r}{V}{W}=\left\{\phi\colon U\to\phi(U)\;\middle|
    \begin{array}{l}
        \mbox{$U\subset V$ and $\phi(U)\subset W$ are both open,} \\
        \mbox{$\phi\colon U\to\phi(U)$ is a diffeomorphism,} \\
        \mbox{$V\cap \closedball{r}\subset U$,\; $W\cap \closedball{r}\subset \phi(U)$.}
    \end{array}
    \right\}.
  \]
  Moreover, for each $\phi\in \diffeomfdemb{r}{V}{W}$, 
  we define $\domain{\phi}{r}=(V\cap \closedball{r})\cup \phi^{-1}(W\cap \closedball{r})$, and 
  \[
    d_r'(\phi)=\sup_{p\in \domain{\phi}{r}}(|p-\phi(p)|+\mu(T_p V,T_{\phi(p)}W)),
  \]
  where $\mu$ is a metric on the Grassmannian $\grassmannian{d}{N}$. If $\domain{\phi}{r}=\emptyset$, 
  we define $d_r'(\phi)=0$.  
  Under this setting, we define
  \[
    d_r(V, W)=\inf_{\phi\in\diffeomfdemb{r}{V}{W}}d_r'(\phi)
  \]
  for each $r>0$. If $\diffeomfdemb{r}{V}{W}=\emptyset$, 
  we set $d_r(V, W)=\infty$. 
  This family $\{d_r\}_{r>0}$ satisfies the axioms in \defref{definition axioms of dr}. 
  The topology induced by this family $\{d_r\}_{r>0}$ coincides with the one defined in \galatius. 
\end{ex}


\begin{proof}
    This result was proved in \metricemb, so we omit the proof here. 
\end{proof}

\section{Tiling spaces and \texorpdfstring{$d_r$}{dr}}

In this section, we recall the definition of tiling spaces and discuss the relationship between the tiling spaces and the family $\{d_r\}_{r>0}$. 
For a more detailed treatment of tiling spaces, we refer the reader to the textbook \tilingtextbook~and the paper \tilinginvertlimit~by Anderson-Putnam. 

\begin{defi}[\tilingtextbook, \tilinginvertlimit, \tilingdflg] \label{definition tiling}
  Let $d\in \mathbb{Z}_{>0}$ be a positive integer. 
  A \textit{tiling} of $\mathbb{R}^d$ is a collection $T\subset \powerset{\mathbb{R}^d}$ that satisfies the following conditions:
  \begin{itemize}
    \item Each $t\in T$ is a translated copy of one of the \textit{prototiles}, where prototiles form a finite set $\prototile=\{\tau_1,\dots,\tau_n\}$ of subsets of $\mathbb{R}^d$. 
    \item Each $t\in T$ is homeomorphic to a closed $d$-disk.
    \item $\bigcup_{t\in T} t =\mathbb{R}^d$.
    \item For each $t_1, t_2\in T$ such that $t_1\neq t_2$, $\Int{t_1}\cap\Int{t_2}= \emptyset$. 
  \end{itemize}
  Let $T$ be a tiling of $\mathbb{R}^d$. For a subset $A\subset\mathbb{R}^d$, we define 
  \[
    T\cap A=\{t\in T\mid t\cap A\neq \emptyset\}.
  \]
\end{defi}

\begin{defi}[\tilingtextbook, \tilinginvertlimit, \tilingdflg] \label{definition tiling space}
  Let $T$ be a tiling of $\mathbb{R}^d$. For each $x\in \mathbb{R}^d$, we define the translated tiling as 
  \[
    T+x=\{t+x\mid t\in T\}. 
  \]
  We then define the orbit of $T$ under the translation action of $\mathbb{R}^d$ as
  \[
    \orbit{T}=\{T+x\mid x\in \mathbb{R}^d\}.
  \]
  For any two tilings $T, T'$, we define $d(T, T')$ as follows:
  \[
    d(T,T')=\inf\left(\left\{\varepsilon>0 \;\middle|
    \begin{array}{l}
        \mbox{There exist $u, v\in \mathbb{R}^d$ such that $|u|<\varepsilon$, $|v|<\varepsilon$,} \\
        \mbox{and $(T+u)\cap \openball{1/\varepsilon}{0}=(T'+v)\cap \openball{1/\varepsilon}{0}$}
    \end{array}
    \right\}\cup \{1/\sqrt{2}\}\right).
  \]
  This $d$ is easily verified to be a metric on $\orbit{T}$. 
  We define the \textit{tiling space} $\tilingsp{T}$ of $T$ as the metric completion of $\orbit{T}$ with respect to $d$.
\end{defi}

\begin{remark}{}{} \label{remark prototiles of tiling space}
   If $T$ is a tiling with $\prototile=\{\tau_1,\dots,\tau_n\}$ as its set of prototiles, then any $T' \in \tilingsp{T}$ can be regarded as a tiling whose set of prototiles is a subset of $\prototile$.
   This is because any limit of a Cauchy sequence in $\orbit{T}$ corresponds to a tiling in which every \textit{patch} appearing in $T'$ can also be found in $T$. 
   Here, a patch refers to a finite and connected union of tiles.
\end{remark}

\begin{prop} \label{proposition definition tiling dr}
    Let $T$ be a tiling of $\mathbb{R}^d$, and let $T_1$ and $T_2$ be two elements of $\tilingsp{T}$. 
    For each $r>0$, we define 
    \[
        d_r(T_1,T_2)=\inf\left\{\varepsilon>0 \;\middle|
        \begin{array}{l}
            \mbox{There exists $u\in \mathbb{R}^d$ such that $|u|<\varepsilon$,} \\
            \mbox{and $(T_1+u)\cap \closedball{r}=T_2\cap \closedball{r}$}
        \end{array}
        \right\}.
    \]
    If there is no $u\in\mathbb{R}^d$ such that $(T_1+u)\cap \closedball{r}=T_2\cap \closedball{r}$, this number is understood to be $\infty$. 
    Then, this family $\{d_r\}_{r>0}$ satisfies the axioms in \defref{definition axioms of dr}. 
\end{prop}

\begin{proof}
    The self-distance axiom is easy. 
    
    We prove upper semi-continuity. 
    Since $r_1\leq r_2$ implies $d_{r_1}(T_1, T_2)\leq d_{r_2}(T_1, T_2)$ straightforwardly, we proceed to prove the second condition. 
    
    Let $r > 0$, $T_1, T_2 \in \tilingsp{T}$, and $\varepsilon > 0$ be arbitrary satisfying $d_r(T_1, T_2) < \varepsilon$. 
    Then, there exists $u\in \mathbb{R}^d$ such that $|u|<\varepsilon$, and $(T_1+u)\cap \closedball{r}=T_2\cap \closedball{r}$. 
    Since the set 
    \[
        \bigcup\{t\in (T_1+u)\cup T_2\mid t\cap \closedball{r}= \emptyset\}
    \]
    is closed, we can separate it from $\closedball{r}$ using the normality of $\mathbb{R}^d$. 
    Therefore, there exists $\delta>0$ such that $(T_1+u)\cap \closedball{(r+\delta)}=T_2\cap \closedball{(r+\delta)}$. 
    This implies $d_{r+\delta}(T_1, T_2)<\varepsilon$. 

    Finally, we prove the weak triangle inequality. 
    Let $T_1, T_2, T_3\in\tilingsp{T}$, and $r_1, r_2, r_3 > 0$ be arbitrary satisfying  
    $d_{r_{12}}(T_1, T_2) < r_3$, $d_{r_{123}}(T_2, T_3) < r_2$. 
    Also, let $d_1, d_2 > 0$ be arbitrary satisfying 
    \[
        d_{r_{12}}(T_1, T_2)<d_1<r_3, \quad d_{r_{123}}(T_2, T_3)<d_2<r_2. 
    \] 
    It suffices to prove that $d_{r_1}(T_1, T_3)\leq d_1+d_2$. 
    First, there exist $u, v\in \mathbb{R}^d$ such that $|u|<d_1$, $|v|<d_2$, and 
    $(T_1+u)\cap \closedball{r_{12}}=T_2\cap \closedball{r_{12}}$, $(T_2+v)\cap \closedball{r_{123}}=T_3\cap \closedball{r_{123}}$. 
    Using these $u$ and $v$, we now verify the following identity to prove $d_{r_1}(T_1, T_3)\leq d_1+d_2$:
    \[
        (T_1+u+v)\cap \closedball{r_1}=T_3\cap \closedball{r_1}.
    \]
    Let $t_1+u+v\in (T_1+u+v)\cap \closedball{r_1}$ be a tile. Since $(t_1+u+v)\cap\closedball{r_1}\neq \emptyset$ and $|v|<d_2<r_2$, it follows that $(t_1+u)\cap \closedball{r_{12}}\neq \emptyset$. 
    Thus, there exists $t_2\in T_2\cap \closedball{r_{12}}$ such that $t_1+u=t_2$.
    Since $t_2+v=t_1+u+v\in (T_1+u+v)\cap \closedball{r_1}$, it follows that $(t_2+v)\cap \closedball{r_1}\neq \emptyset$, which in turn implies $(t_2+v)\cap \closedball{r_{123}}\neq \emptyset$. 
    Therefore, there exists $t_3\in T_3\cap \closedball{r_{123}}$ such that $t_2+v=t_3$. 
    Using this $t_3$, we obtain 
    \[
        t_1+u+v=t_3\in T_3\cap\closedball{r_1}, 
    \]
    which implies $(T_1+u+v)\cap \closedball{r_1}\subset T_3\cap \closedball{r_1}$. 

    Conversely, let $t_3\in T_3\cap \closedball{r_1}$ be a tile. 
    Since $(T_2+v)\cap\closedball{r_{123}}=T_3\cap \closedball{r_{123}}$, there exists $t_2\in T_2$ such that $t_2+v=t_3$. 
    From $t_2=t_3-v$, $|v|<d_2<r_2$, and $t_3\cap\closedball{r_1}\neq \emptyset$, it follows that $t_2\cap\closedball{r_{12}}\neq \emptyset$. 
    Therefore, there exists $t_1\in T_1$ such that $t_1+u=t_2$. 
    Using this $t_1$, we have 
    \[
        t_3=t_1+u+v\in (T_1+u+v)\cap\closedball{r_1}, 
    \]
    which implies $T_3\cap \closedball{r_1}\subset (T_1+u+v)\cap \closedball{r_1}$. 
    Thus, we conclude that $(T_1+u+v)\cap \closedball{r_1}=T_3\cap \closedball{r_1}$. 
    
    From this identity and the fact that $|u+v|<d_1+d_2$, it follows that $d_{r_1}(T_1, T_3)<d_1+d_2$. 
    Since $d_1$ and $d_2$ are arbitrary values satisfying $d_{r_{12}}(T_1, T_2)<d_1<r_3$ and $d_{r_{123}}(T_2, T_3)<d_2<r_2$, we conclude that $d_{r_1}(T_1, T_3)\leq d_{r_{12}}(T_1, T_2)+d_{r_{123}}(T_2, T_3)$.
\end{proof}

\begin{prop} \label{proposition tiling topology}
    For any tiling $T$, the tiling space $\tilingsp{T}$ equipped with the topology induced by the family $\{d_r\}_{r>0}$ (\proref{proposition definition tiling dr}) is homeomorphic to 
    $\tilingsp{T}$ endowed with the metric topology of $d$ (\defref{definition tiling space}).
\end{prop}

\begin{proof}
    Let $T_1$ be a tiling in $\tilingsp{T}$. For any $r>0$ and $\varepsilon>0$, we need to find a number $\varepsilon'>0$ such that $\openball{\varepsilon'}{T_1}\subset \nhddr{r}{\varepsilon}{T_1}$, 
    where $\openball{\varepsilon'}{T_1}=\{T'\in\tilingsp{T}\mid d(T_1, T')<\varepsilon'\}$.
    Let $\prototile$ be the set of all prototiles of $T$. By \remref{remark prototiles of tiling space}, any tiling in $\tilingsp{T}$ consists of tiles that are translated copies of the prototiles in $\prototile$. 
    We define $\varepsilon'>0$ to be a positive number satisfying 
    \[
        \varepsilon'<\min\left\{\dfrac{1}{r+\varepsilon/2+\max_{\tau\in \prototile}\diam{\tau}}, \varepsilon/2\right\}. 
    \]
    Then, the inclusion $\openball{\varepsilon'}{T_1}\subset \nhddr{r}{\varepsilon}{T_1}$ holds. 
    Indeed, let $T'\in \openball{\varepsilon'}{T_1}$ be an element. 
    Then, there exist $u, v\in\mathbb{R}^d$ such that $|u|<\varepsilon'$, $|v|<\varepsilon'$, and 
    \[
        (T_1+u)\cap \openball{1/\varepsilon'}{0}=(T'+v)\cap \openball{1/\varepsilon'}{0}.
    \] 
    Since $|u-v|<2\varepsilon'<\varepsilon$, there exists $w\in \mathbb{R}^d$ such that $|w|<\varepsilon$ and 
    \[
        (T_1+w)\cap \openball{1/\varepsilon'- \varepsilon/2}{0} = T'\cap \openball{1/\varepsilon'-\varepsilon/2}{0}.
    \]
    From $r+\varepsilon/2+\max_{\tau\in \prototile}\diam{\tau}<1/\varepsilon'$, it follows that 
    \[
        (T_1+w)\cap \closedball{r} = T'\cap \closedball{r}, 
    \]
    which implies $d_r(T_1, T')<\varepsilon$. Thus, we conclude that $\openball{\varepsilon'}{T_1}\subset \nhddr{r}{\varepsilon}{T_1}$. 
    
    Next, we prove the reverse inclusion. 
    Let $T_1$ be a tiling in $\tilingsp{T}$. For any $\varepsilon>0$ such that $\varepsilon<1/\sqrt{2}$, we need to find two numbers $r'>0$ and $\varepsilon'>0$ such that $\nhddr{r'}{\varepsilon'}{T_1}\subset \openball{\varepsilon}{T_1}$. 
    Define $r'>0$ and $\varepsilon'>0$ to satisfy $\varepsilon'<\varepsilon$ and $r'>1/\varepsilon$. Then, $\nhddr{r'}{\varepsilon'}{T_1}\subset \openball{\varepsilon}{T_1}$. 
    Indeed, let $T'\in \nhddr{r'}{\varepsilon'}{T_1}$ be an element. 
    Then, there exists $w\in \mathbb{R}^d$ such that $|w|<\varepsilon'<\varepsilon$ and 
    \[
        (T_1+w)\cap \closedball{r'}=T'\cap \closedball{r'}. 
    \]
    Since $\openball{1/\varepsilon}{0}\subset \closedball{r'}$, we obtain 
    \[
        (T_1+w)\cap \openball{1/\varepsilon}{0}=T'\cap \openball{1/\varepsilon}{0},  
    \]
    which implies $d(T_1, T')<\varepsilon$. Thus, we conclude that $\nhddr{r'}{\varepsilon'}{T_1}\subset \openball{\varepsilon}{T_1}$. 
\end{proof}


\section{Uniformity defined by \texorpdfstring{$d_r$}{dr}}

The family $\{d_r\}_{r>0}$ induces not only a topology but also a quasi-uniformity. Moreover, an even weaker version of the triangle inequality suffices to define the quasi-uniformity.
In particular, this weaker triangle inequality is satisfied in the space of graph embeddings $\embgraph{\mathbb{R}^N}$, as defined in \graphemb. 



Before proceeding, we provide a brief introduction to quasi-uniform spaces, which generalize uniform spaces by not necessarily requiring symmetry. 
For a more detailed treatment of uniform spaces, see \broubakione, Chapter II by Bourbaki, and for quasi-uniform spaces, see \quasiuniform~by Fletcher-Lindgren.

\begin{defi}[\broubakione, Chapter II, \quasiuniform, Chapter 1] \label{definition uniform space}
    Let $X$ be a set. A \textit{uniformity} on $X$ is a collection $\mathcal{U}$ of subsets of $X\times X$ (called \textit{entourages}) satisfying the following five axioms:
    \begin{itemize}
        \item If $U\in \mathcal{U}$ and $U\subset V\subset X\times X$, then $V\in \mathcal{U}$. 
        \item If $U, V\in \mathcal{U}$, then $U\cap V\in \mathcal{U}$. 
        \item If $U\in \mathcal{U}$, then $\Delta\subset U$, where $\Delta=\{(x,x)\in X\times X\mid x\in X\}$ is the diagonal. 
        \item If $U\in \mathcal{U}$, then there exists $V\in \mathcal{U}$ such that $V\circ V\subset U$, where for any two subsets $U, V\subset X\times X$, we define 
        \[
        U\circ V=\{(x,z)\mid \text{there exists $y\in X$ such that $(x,y)\in U$ and $(y,z)\in V$.}\}.
        \]
        \item If $U\in\mathcal{U}$, then $U^{-1}\in \mathcal{U}$, where $U^{-1}=\{(y,x)\in X\times X\mid (x,y)\in U\}\in \mathcal{U}$ for $U\subset X\times X$. 
        \end{itemize}
    A \textit{uniform space} is a set equipped with a uniformity. If the last condition is omitted, $\mathcal{U}$ is called a \textit{quasi-uniformity}, and a set equipped with a quasi-uniformity is called a \textit{quasi-uniform space}. 
\end{defi}

\begin{prop}[\broubakione, Chapter II, \S 1.2, Proposition 1] \label{proposition nhd system uniform space}
  Let $(X,\mathcal{U})$ be a quasi-uniform space. For each $U\in \mathcal{U}$ and $x\in X$, define 
  \[
    U[x]=\{y\in X\mid (x,y)\in U\}, 
  \]
  and let $\mathcal{B}(x)=\{U[x]\mid U\in \mathcal{U}\}$. Then, $\mathcal{B}(x)$ forms a neighborhood system for some topology on $X$, i.e., it satisfies the following conditions:
    \begin{itemize}
      \item For each $x\in X$ and each $U, V\subset X$ satisfying $V\in \mathcal{B}(x)$ and $V\subset U$, we have $U\in \mathcal{B}(x)$. 
      \item For each $x\in X$ and $U_1, U_2\in\mathcal{B}(x)$, we have $U_1\cap U_2\in \mathcal{B}(x)$. 
      \item For each $x\in X$ and $U\in \mathcal{B}(x)$, we have $x\in U$. 
      \item For each $x\in X$ and $U\in \mathcal{B}(x)$, there exists $V\in \mathcal{B}(x)$ such that $U\in\mathcal{B}(y)$ for all $y\in V$. 
    \end{itemize}
\end{prop}

\begin{proof}
    See \broubakione, Chapter II, \S 1.2, Proposition 1. 
\end{proof}

\begin{remark}
    Although \broubakione~states this proposition for uniform spaces, the proof does not rely 
    on the symmetry of the entourages. Thus, the statement also holds for quasi-uniform spaces.
\end{remark}


    
We next introduce a still weaker form of the triangle inequality, which generalizes the weak triangle inequality of \defref{definition axioms of dr}.

\begin{defi} \label{definition weaker triangle inequality}
    Let $\{d_r\}_{r>0}$ be a family satisfying the axioms of \defref{definition axioms of dr}.
    Instead of the weak triangle inequality in \defref{definition axioms of dr}, we assume the following weaker triangle inequality:
    
    \begin{description}
        \item[Weaker triangle inequality]\leavevmode\\
        Let $r_1, r_2, r_3 > 0$. If $d_{r_1+r_2+r_3}(x, y)<r_3$ and $d_{r_1+r_2+r_3}(y, z)<r_2$, then 
        \[
            d_{r_1}(x, z)\leq d_{r_1+r_2+r_3}(x, y)+d_{r_1+r_2+r_3}(y, z).
        \]
    \end{description}
    A set $(X, \{d_r\}_{r>0})$ is said to satisfy the weaker $\{d_r\}$-axioms if it satisfies the axioms in \defref{definition axioms of dr}, 
    except that the weak triangle inequality is replaced by the weaker triangle inequality. 
\end{defi}

\begin{remark}
    The key difference is that in the weak triangle inequality, the term $d_{r_1+r_2}(x, y)$ appears, whereas in the weaker triangle inequality, this is replaced by $d_{r_1+r_2+r_3}(x, y)$.
    The weaker triangle inequality follows from the weak triangle inequality since $d_{r_1+r_2}(x, y) \leq d_{r_1+r_2+r_3}(x, y)$. 

    In \exref{example graphemb}, we will define a family $\{d_r\}_{r>0}$ on the space of graph embeddings $\embgraph{\mathbb{R}^N}$ that satisfies the weaker triangle inequality.
    While the weak triangle inequality may also hold in this space, we are unable to verify it directly.
\end{remark}

The following definition provides the quasi-uniformity induced by the family $\{d_r\}_{r>0}$. 

\begin{prop} \label{proposition definition uniformity of dr space}
    Let $(X, \{d_r\}_{r>0})$ satisfy the weaker $\{d_r\}$-axioms in \defref{definition weaker triangle inequality}.
    We define
    \[
    \entourage{r}{\varepsilon}=\{(x,y)\in X\times X\mid d_r(x,y)<\varepsilon\}. 
    \]
    A subset $U\subset X\times X$ is called an entourage if there exist $r>0$ and $\varepsilon>0$ such that $\entourage{r}{\varepsilon}\subset U$. 

    Then, the collection $\mathcal{U}$ of all such entourages forms a quasi-uniformity on $X$.
    Moreover, if the family $\{d_r\}_{r>0}$ satisfies the symmetry condition, i.e., 
    \[
        d_r(x,y)=d_r(y,x)
    \]
    for all $x, y\in X$ and $r>0$, then in this case, $\mathcal{U}$ forms a uniformity. 
\end{prop}

\begin{proof}
    We verify that the collection $\mathcal{U}$ satisfies the four conditions required for a quasi-uniformity.

    First, let $U \in \mathcal{U}$ and let $V \subset X \times X$ be any subset satisfying $U \subset V$. 
    By definition of $\mathcal{U}$, there exist $r > 0$ and $\varepsilon > 0$ such that $\entourage{r}{\varepsilon} \subset U$. 
    Since $U \subset V$, it follows that $\entourage{r}{\varepsilon} \subset V$. 
    By definition of $\mathcal{U}$, this implies that $V \in \mathcal{U}$.

    Next, we verify that $\mathcal{U}$ is closed under intersections.
    Let $U_1, U_2 \in \mathcal{U}$ be two entourages. 
    By definition, there exist $r_1, r_2 > 0$ and $\varepsilon_1, \varepsilon_2 > 0$ such that 
    \[
    \entourage{r_1}{\varepsilon_1} \subset U_1 \text{ and } \entourage{r_2}{\varepsilon_2} \subset U_2.
    \]
    Thus, we have $\entourage{r_1+r_2}{\min\{\varepsilon_1, \varepsilon_2\}}\subset\entourage{r_1}{\varepsilon_1}\cap \entourage{r_2}{\varepsilon_2}\subset U_1\cap U_2$. 
    It follows that $U_1 \cap U_2 \in \mathcal{U}$.

    Now, we show that each entourage contains the diagonal.
    Let $U \in \mathcal{U}$. There exist $r>0$ and $\varepsilon>0$ such that $\entourage{r}{\varepsilon}\subset U$. 
    Since $d_r(x,x)=0<\varepsilon$ for all $x\in X$, we have $(x,x)\in \entourage{r}{\varepsilon}\subset U$. 

    Finally, we verify the last condition.  
    Let $U \in \mathcal{U}$. There exist $r>0$ and $\varepsilon>0$ such that $\entourage{r}{\varepsilon}\subset U$. 
    Set $r' = 3r$, $\varepsilon' = \min\{\varepsilon/2, r\}$, and $V=\entourage{r'}{\varepsilon'}\in \mathcal{U}$. 
    Thus, we have 
    \[
    V\circ V\subset U.
    \]
    Indeed, let $(x,z)\in V\circ V$ be arbitrary. There exists $y\in X$ such that $(x,y)\in V$ and $(y,z)\in V$. 
    By definition of $V$, we have $d_{r'}(x,y)<\varepsilon'$ and $d_{r'}(y,z)<\varepsilon'$. 
    By these inequalities, we have 
    \[
    d_{2r+\varepsilon'}(x,y)\leq d_{r'}(x,y)<\varepsilon'\leq r, \text{ and } d_{2r+\varepsilon'}(y,z)\leq d_{r'}(y,z)<\varepsilon'.
    \]
    Therefore, by the weaker triangle inequality, we have 
    \[
    d_r(x,z)\leq d_{2r+\varepsilon'}(x,y)+d_{2r+\varepsilon'}(y,z)<\varepsilon/2+\varepsilon/2=\varepsilon,  
    \]
    which implies that $(x,z)\in \entourage{r}{\varepsilon}\subset U$. 

    Moreover, we show that if symmetry is assumed, then the quasi-uniformity $\mathcal{U}$ is in fact a uniformity. 
    Let $U \in \mathcal{U}$. There exist $r > 0$ and $\varepsilon > 0$ such that $\entourage{r}{\varepsilon} \subset U$. 
    For each $(y, x) \in \entourage{r}{\varepsilon}$, we have $d_r(y, x) = d_r(x, y) < \varepsilon$. 
    Therefore, it follows that $(y, x) \in U^{-1}$, which implies that $U^{-1} \in \mathcal{U}$.
\end{proof}

As an example, we show how the family $\{d_r\}_{r>0}$, satisfying the weaker triangle inequality, can describe the \textit{$C^0$-topology} on the space $\embgraph{\mathbb{R}^N}$ of all graphs in $\mathbb{R}^N$, as defined in \graphemb. 
For consistency, we follow the definitions and notation in \graphemb~without restating them here. 


\begin{ex} \label{example graphemb}
    Let $\embgraph{\mathbb{R}^N}$ be the space of all graphs $G$ defined in \graphemb. 
    For each $r>0$ and $G, H\in \embgraph{\mathbb{R}^N}$, we define 
    \[
      \diffeomfdemb{r}{G}{H}=\left\{\varphi\colon U\to\varphi(U)\;\middle|
      \begin{array}{l}
          \mbox{$U\subset H$ and $\varphi(U)\subset G$ are both open,} \\
          \mbox{$(U, \varphi(U), \varphi)$ is a triple (defined in \graphemb, Definition 2.4),} \\
          \mbox{$H\cap \closedball{r}\subset U$, $G\cap \closedball{r}\subset \varphi(U)$.}
      \end{array}
      \right\}.
    \]
    Moreover, for each $\varphi\in \diffeomfdemb{r}{G}{H}$, 
    we define $\domain{\varphi}{r}=(H\cap \closedball{r})\cup \varphi^{-1}(G\cap \closedball{r})$, and 
    \[
      d_r'(\varphi)=\sup_{p\in \domain{\varphi}{r}}|p-\varphi(p)|.
    \]
    If $\domain{\varphi}{r}=\emptyset$, 
    we set $d'_r(\varphi) = 0$. 
    Under this setting, we define 
    \[
      d_r(G, H)=\inf_{\varphi\in\diffeomfdemb{r}{G}{H}}d_r'(\varphi)
    \]
    for each $r>0$. If $\diffeomfdemb{r}{G}{H}=\emptyset$, 
    we set $d_r(G, H) = \infty$. 
    This family $\{d_r\}_{r>0}$ satisfies the weaker $\{d_r\}$-axioms in \defref{definition weaker triangle inequality}.
    The topology induced by $\{d_r\}_{r>0}$ coincides with the $C^0$-topology defined in \graphemb, Definition 2.5. 
\end{ex}

\begin{remark}{}{} \label{remark the order of triple}
    In this example, the family $\{d_r\}_{r>0}$ may not satisfy the symmetry condition $d_r(G, H)=d_r(H, G)$. 
    Moreover, note that the domain of the triple $\varphi$ in $\diffeomfdemb{r}{G}{H}$ is an open subset of $H$, not $G$. 
    Additionally, \graphemb~defines a \textit{$C^1$-topology} for the space $\embgraph{\mathbb{R}^N}$. 
    However, regardless of how one defines the family $\{d_r\}_{r>0}$, it seems difficult to recover this topology from it.
\end{remark}

\begin{proof}
    Although the proof is largely similar to Lemma 4.2, in \metricemb, some conditions have been modified. Therefore, we provide the proof here for completeness. 

    The self-distance axiom is immediate. 
    
    We now prove the upper semi-continuity. 
    Since $r_1\leq r_2$ implies $d_{r_1}(G, H)\leq d_{r_2}(G, H)$ straightforwardly, we proceed to prove the other condition. 

    Let $r>0$, $G, H\in \embgraph{\mathbb{R}^N}$, and $\varepsilon>0$ be arbitrary satisfying $d_r(G, H)<\varepsilon$. 
    Then, there exists a map $\varphi\in\diffeomfdemb{r}{G}{H}$ with $\varphi\colon U\to \varphi(U)$ satisfying $|p-\varphi(p)|<\varepsilon$ for all $p\in \domain{\varphi}{r}$. 
    Since $U$ is open in $H$, the set 
    \[
    A=\{p\in U\mid |p-\varphi(p)|<\varepsilon\} 
    \]
    is also open in $H$. This open set contains $\domain{\varphi}{r}=(H\cap \closedball{r})\cup \varphi^{-1}(G\cap \closedball{r})$. 
    Therefore, there exist $\delta_1>0$ and $\delta_2>0$ such that $H\cap \closedball{(r+\delta_1)}\subset A$ and $G\cap \closedball{(r+\delta_2)}\subset\varphi(A)$. 
    Setting $\delta=\min\{\delta_1,\delta_2\}$, we obtain $\domain{\varphi}{r+\delta}\subset A$, which implies $\varphi \in \diffeomfdemb{r+\delta}{G}{H}$. 
    Thus, 
    \[
        d_{r+\delta}(G, H)\leq d_{r+\delta}'(\varphi)<\varepsilon. 
    \]

    Next, we prove the weaker triangle inequality. 
    Let $G_1, G_2, G_3\in\Phi(\mathbb{R}^N)$, and $r_1, r_2, r_3 > 0$ be arbitrary. 
    Suppose that  
    \[
        d_{r_{123}}(G_1, G_2) < r_3, \quad d_{r_{123}}(G_2, G_3) < r_2.
    \]
    To establish the weaker triangle inequality, let $ d_3, d_2 > 0 $ be arbitrary such that  
    \[
        d_{r_{123}}(G_1, G_2)<d_3<r_3, \quad d_{r_{123}}(G_2,G_3)<d_2<r_2.
    \]
    We need to show that  
    \[
        d_{r_1}(G_1,G_3)\leq d_3+d_2.
    \]
    There exist two maps 
    $(\varphi_1\colon U_1\to \varphi_1(U_1))\in \diffeomfdemb{r_{123}}{G_1}{G_2}$ and $(\varphi_2\colon U_2\to \varphi_2(U_2))\in \diffeomfdemb{r_{123}}{G_2}{G_3}$ 
    such that 
    \[
        |p_2-\varphi_1(p_2)|<d_3 \mbox{ for all } p_2\in \domain{\varphi_1}{r_{123}},
    \]
    \[
        |p_3-\varphi_2(p_3)|<d_2 \mbox{ for all } p_3\in\domain{\varphi_2}{r_{123}}.
    \]
    For any $p_3\in\domain{\varphi_2}{r_{13}}$, we have either $p_3\in G_3\cap \closedball{r_{13}}$ or $\varphi_2(p_3)\in G_2\cap \closedball{r_{13}}\subset G_2\cap\closedball{r_{123}}$. 
    Even if $p_3\in G_3\cap \closedball{r_{13}}$, the inequality $|p_3-\varphi_2(p_3)|<d_2<r_2$ implies that $\varphi_2(p_3)\in G_2 \cap\closedball{r_{123}}$. 
    Therefore, for any $p_3\in\domain{\varphi_2}{r_{13}}$, we obtain $\varphi_2(p_3)\in G_2\cap \closedball{r_{123}}$, 
    which allows us to restrict the domain of $\varphi_2$ to a smaller neighborhood $U$ that contains $\domain{\varphi_2}{r_{13}}$ and on which the composition $\varphi_1 \circ \varphi_2$ is well-defined.
    
    For this composition $\varphi_1\circ \varphi_2$, we have $\varphi_1\circ \varphi_2\in \diffeomfdemb{r_1}{G_1}{G_3}$. 
    Indeed, let $p_1\in G_1\cap \closedball{r_1}\subset G_1\cap \closedball{r_{123}}\subset\varphi_1(U_1)$ be arbitrary. 
    There exists $p_2\in U_1$ such that $p_1=\varphi_1(p_2)$. 
    Since $|\varphi_1(p_2)-p_2|<d_3<r_3$ holds, we have $p_2\in G_2\cap \closedball{r_{13}}\subset \varphi_2(U)$. 
    Therefore, we obtain $G_1\cap \closedball{r_1}\subset \varphi_1\circ \varphi_2(U)$. 

    For this composition $\varphi_1\circ \varphi_2$, we have 
    \[
        |p_3-\varphi_1\circ \varphi_2(p_3)|\leq |\varphi_1\circ\varphi_2(p_3)-\varphi_2(p_3)|+|p_3-\varphi_2(p_3)|\leq d_3+d_2, 
    \]
    for any $p_3\in \domain{\varphi_1\circ \varphi_2}{r_1}$. This implies that 
    \[
        d_{r_1}(G_1,G_3)\leq d_{r_1}'(\varphi_1\circ \varphi_2)\leq d_3+d_2. 
    \]

    Finally, we prove that the topology induced by $\{d_r\}_{r>0}$ coincides with the $C^0$-topology defined in \graphemb, Definition 2.5. 
    First, we show that for any $G\in\embgraph{\mathbb{R}^N}$, any compact set $K\subset \mathbb{R}^N$, and any $\varepsilon>0$, there exist $r'>0$ and $\varepsilon'>0$ such that 
    \[
      \entourage{r'}{\varepsilon'}[G]\subset \nhdgalatius{K}{\varepsilon}{G}. 
    \]
    Since $K$ is compact, there exists $r'>0$ such that $K\subset \closedball{r'}$. 
    Setting $\varepsilon'=\varepsilon$, we obtain the desired inclusion. 
    Indeed, for any $H\in \entourage{r'}{\varepsilon'}[G]$, 
    there exists $\varphi\colon U\to \phi(U)$ in $\diffeomfdemb{r'}{G}{H}$  
    such that $|p-\varphi(p)|<\varepsilon$ for all $p\in \domain{\varphi}{r'}$. 
    Since $\varphi\in\diffeomfdemb{r'}{G}{H}$, we have $K\cap G\subset \closedball{r'}\cap G\subset \varphi(U)$. 
    Furthermore, for any $p\in K^{\varepsilon} \cap H\subset \closedball{r'}\cap H$, we obtain $|\varphi(p)-p|<\varepsilon'=\varepsilon$, which implies that $\varphi(p)\in K$. 
    Therefore, we conclude that $K^{\varepsilon} \cap H\subset \varphi^{-1}(K)$. 
    Thus, $(U, \varphi(U),\varphi)$ is a triple $\varphi \colon H \dashrightarrow G$ that is $(\varepsilon, K)$-small, implying that $H\in \nhdgalatius{K}{\varepsilon}{G}$. 

    Moreover, we show that for any $G\in\embgraph{\mathbb{R}^N}$, any $r>0$ and $\varepsilon>0$, there exist a compact set $K\subset \mathbb{R}^N$ and a number $\varepsilon'>0$ such that  
    \[
        \nhdgalatius{K}{\varepsilon'}{G}\subset \entourage{r}{\varepsilon}[G]. 
    \]
    Set $K=\closedball{(r+\varepsilon)}$ and $\varepsilon'=\varepsilon$. Then, we obtain the desired inclusion. 
    Indeed, for any $H\in \nhdgalatius{K}{\varepsilon'}{G}$, there exists a triple $(V',V ,\varphi)$ that is $(\varepsilon, K)$-small, meaning that 
    \[
        K\cap G\subset V, \quad K^{\varepsilon'}\cap H\subset \varphi^{-1}(K), \mbox{ and } |p-\varphi(p)|<\varepsilon' \mbox{ for all } p\in \varphi^{-1}(K).
    \]
    For the given triple $(V',V ,\varphi)$, we obtain 
    \[
        K\cap G\subset V=\varphi(V'), \mbox{ and } \closedball{r}\cap H\subset K^{\varepsilon'}\cap H\subset \varphi^{-1}(K)\subset V'.
    \]
    Therefore, this triple belongs to $\diffeomfdemb{r}{G}{H}$ with $d_r'(\varphi)<\varepsilon'=\varepsilon$, which implies that $H\in \entourage{r}{\varepsilon}[G]$. 
\end{proof}

The family $\{d_r\}_{r>0}$ induces a pseudo-metric that is not necessarily symmetric. 

\begin{defi} \label{definition non-symmetric pseudo-metric}
  Let $X$ be a set. A \textit{quasi-pseudo-metric} on $X$ is a function $d\colon X\times X\to \mathbb{R}_{\geq 0}$ satisfying the following:
  \begin{description}
    \item[Self-distance Axiom] $d(x,x)=0$ for all $x\in X$. 
    \item[Triangle inequality] $d(x, z)\leq d(x, y)+d(y, z)$ for all $x, y, z\in X$.  
  \end{description}
  If the following condition is also satisfied, then we call $d$ a \textit{quasi-metric}: 
  \begin{description}
    \item[Non-degeneracy] $d(x, y)=0$ implies $x=y$.  
  \end{description}

  Given a quasi-pseudo-metric $d$ on a set $X$, we define
  \[
    V_{\varepsilon}=\{(x,y)\in X\times X\mid d(x, y)< \varepsilon\}
  \]
  for each $\varepsilon>0$. A subset $U\subset X\times X$ is called an entourage if there exists $\varepsilon>0$ such that $V_{\varepsilon}\subset U$. 
  Then, the collection $\mathcal{U}$ of all such entourages forms a quasi-uniformity on $X$. 
\end{defi}

\begin{thm} \label{theorem pseudo-metric of dr space}
    Let $(X, \{d_r\}_{r>0})$ satisfy the weaker $\{d_r\}$-axioms in \defref{definition weaker triangle inequality}. 
    Then there exists a quasi-pseudo-metric $d$ on $X$ that induces the quasi-uniformity defined in \proref{proposition definition uniformity of dr space}. 
    Moreover, if the family $\{d_r\}_{r>0}$ is non-degenerate, i.e., 
    \[
        d_r(x, y)=0 \text{ for all } r>0\implies x=y, 
    \]
    then $d$ can be chosen to be a quasi-metric on $X$. 
\end{thm}

\begin{proof}
    The following argument is based on Proposition 2 in Chapter IX, Section 1.4 of \broubakitwo~by Bourbaki. 
    
    We define $W_n=\entourage{3^{2n}}{3^{-n}}=\{(x,y)\in X\times X\mid d_{3^{2n}}(x,y)<3^{-n}\}$ for each $n\in \mathbb{Z}_{\geq 0}$. 
    We now show that the following inclusion holds for each $n\in \mathbb{Z}_{\geq 0}$: 
    \[
        W_{n+1}\circ W_{n+1}\circ W_{n+1}\subset W_n. 
    \]
    Let $(x, w)\in W_{n+1}\circ W_{n+1}\circ W_{n+1}$ be arbitrary. 
    Then, there exist $y, z\in X$ such that $(x, y)$, $(y, z)$, and $(z, w)$ are all in $W_{n+1}$. 
    By applying the weaker triangle inequality repeatedly, we obtain
    \[
    \begin{aligned}
        &d_{3^{2n}}(x, w)\leq d_{3^{2n+1}}(x, y) + d_{3^{2n+1}}(y, w)\leq d_{3^{2n+2}}(x, y) + d_{3^{2n+1}}(y, w)\\
        &\leq d_{3^{2n+2}}(x, y) + d_{3^{2n+2}}(y, z) + d_{3^{2n+2}}(z, w)\leq 3^{-n-1}+3^{-n-1}+3^{-n-1}=3^{-n}. 
    \end{aligned}
    \]
    Using these $\{W_n\}_{n\in \mathbb{Z}_{\geq 0}}$, we define a function $g\colon X\times X\to \mathbb{R}_{\geq 0}$ by
    \[
        g(x, y) = 
        \begin{cases}
            0 & \text{if } (x, y) \in W_n \text{ for all } n \geq 1, \\
            2^{-k} & \text{if } (x, y) \in W_n \text{ for } 1 \leq n \leq k \text{ and } (x, y) \notin W_{k+1}, \\
            1 & \text{if } (x, y) \notin W_1.
        \end{cases}
    \]
    for $(x, y)\in X\times X$. Also, we define a function $d\colon X\times X\to \mathbb{R}_{\geq 0}$ by 
    \[
        d(x, y)=\inf \sum_{i=0}^{p-1} g(z_i, z_{i+1}), 
    \]
    where the infimum is taken over all finite sequences $\{z_i\}_{0\leq i\leq p}$ such that $z_0=x$ and $z_p=y$. 
    We shall show that $d$ is a quasi-pseudo-metric on $X$ which satisfies the following inequalities:
    \[
        \dfrac{1}{2}\,g(x, y)\leq d(x, y)\leq g(x, y).
    \]

        
    It is straightforward to verify that $d$ is a quasi-pseudo-metric on $X$.  
    Since the inequality $d(x, y) \leq g(x, y)$ follows directly from the definition, we only need to prove the inequality $\frac{1}{2} g(x, y)\leq d(x, y)$.  
    To this end, we prove the following by induction on $p \in \mathbb{Z}_{>0}$:

For any $x, y\in X$ and any finite sequence $\{z_i\}_{0 \leq i \leq p}$ in $X$ with $z_0 = x$ and $z_p = y$, we have
\[
    \dfrac{1}{2}\,g(x, y)\leq \sum_{i=0}^{p-1} g(z_i, z_{i+1}).
\]
The case $p=1$ is immediate. Suppose that the statement holds for every $p' < p$. We now prove the case $p$.
We define $a=\sum_{i=0}^{p-1} g(z_i, z_{i+1})$. If $a\geq 1/2$, then the desired inequality holds trivially.
If $a=0$, the claim is clear. Thus assume $0<a<1/2$.

If there exists $q \in \{1,\dots,p-1\}$ such that
\[
    \sum_{i=0}^{q-1} g(z_i, z_{i+1}) \leq \dfrac{a}{2},
\]
let $h$ be the largest such integer; otherwise, set $h=0$.
Then
\[
    \sum_{i=0}^{h-1}g(z_i, z_{i+1})\leq\dfrac{a}{2}\quad \text{and} \quad \sum_{i=0}^h g(z_i, z_{i+1})>\dfrac{a}{2},
\]
where the first sum is understood to be $0$ when $h=0$.
It follows that $\sum_{i=h+1}^{p-1} g(z_i, z_{i+1})< a/2$.

Applying the inductive hypothesis to each nontrivial one of the two subsequences
$z_0,\dots,z_h$ and $z_{h+1},\dots,z_p$ (and using $g(x,x)=g(y,y)=0$ in the trivial case),
we obtain $g(x, z_h)\leq a$ and $g(z_{h+1}, y)\leq a$. It is also clear that $g(z_h, z_{h+1})\leq a$.

Let $k\in\mathbb{Z}_{\geq 1}$ be the smallest integer such that $2^{-k}\leq a$.
By the definition of $g$, it follows that $(x, z_h),\ (z_h, z_{h+1})$, and $(z_{h+1}, y)$ all belong to $W_k$.
Therefore, we obtain 
\[
    (x, y)\in W_k\circ W_k\circ W_k\subset W_{k-1},
\]
which implies that $g(x, y)\leq 2^{1-k}\leq 2a$.

    Using these inequalities, we now show that the quasi-uniformity $\mathcal{U}$ induced by the family $\{d_r\}_{r>0}$ coincides with the quasi-uniformity $\mathcal{U}'$ induced by $d$. 
    Let $U\in \mathcal{U}$ be an arbitrary entourage. 
    Then there exist $r>0$ and $\varepsilon>0$ such that $\entourage{r}{\varepsilon}\subset U$. 
    For these $r$ and $\varepsilon$, choose $k\in\mathbb{Z}_{\geq 0}$ such that $3^{2k}>r$, $3^{-k}<\varepsilon$. 
    Since $g(x, y)\leq 2d(x, y)$ for any $x, y\in X$, we have 
    \[
        V_{2^{-k-1}}\subset W_k\subset \entourage{r}{\varepsilon}\subset U. 
    \]
    Indeed, for any $(x, y)\in V_{2^{-k-1}}$, we have $g(x, y)\leq 2d(x, y)< 2^{-k}$, which implies $(x, y)\in W_k$. 
    Therefore, $U$ is also in $\mathcal{U}'$. 

    Conversely, let $U\in \mathcal{U}'$ be an arbitrary entourage. 
    Then there exists $\varepsilon>0$ such that $V_{\varepsilon}\subset U$. Choose $k\in\mathbb{Z}_{\geq 0}$ such that $2^{-k}< \varepsilon$. 
    Since $d(x, y)\leq g(x, y)$ for any $x, y\in X$, we have
    \[
        \entourage{3^{2k}}{3^{-k}}=W_k\subset V_{\varepsilon}.
    \]
    Indeed, for any $(x, y)\in W_k$, we have $d(x, y)\leq g(x, y)\leq 2^{-k}<\varepsilon$. Therefore, $U$ is also in $\mathcal{U}$. 

    Finally, we prove that $d(x, y) = 0$ implies $x = y$ under the assumption that the family $\{d_r\}_{r>0}$ is non-degenerate. 
    Let $x, y\in X$ be arbitrary with $d(x, y)=0$, and let $r>0$ and $\varepsilon>0$ be arbitrary. Choose $k\in \mathbb{Z}_{\geq 1}$ such that $W_k\subset \entourage{r}{\varepsilon}$. 
    Since $g(x, y)\leq 2d(x, y)=0$, we obtain $(x, y)\in W_n$ for all $n\in\mathbb{Z}_{\geq 1}$, and in particular $(x, y)\in W_k$.  
    Hence, $(x, y)\in \entourage{r}{\varepsilon}$, it follows that $d_r(x, y)< \varepsilon$. 
    Since $r$ and $\varepsilon$ were arbitrary, this means $d_r(x, y)=0$ for all $r>0$. 
    By the non-degeneracy of $\{d_r\}_{r > 0}$, we conclude that $x = y$.
\end{proof}

\begin{remark}{}{} \label{remark symmetric metric}
    If each $d_r$ is symmetric, then the family $\{d_r\}_{r>0}$ induces a genuine uniformity rather than just a quasi-uniformity.
    In this case, Proposition 2 in Chapter IX, Section 1.4 of \broubakitwo~applies directly and ensures the existence of a pseudo-metric $d$ that induces the original uniformity.
    Moreover, if each $d_r$ is symmetric, then the topology induced by the family $\{d_r\}_{r>0}$ satisfies the $T_1$ separation axiom. In fact, it is even regular (\metricemb, Lemma 4.3).
    As a consequence, the pseudo-metric $d$ constructed in the proof must actually be a metric. 
\end{remark}

\begin{ex} \label{example submanifold metrizable}
  The family $\{d_r\}_{r>0}$ on the space of submanifolds $\embmfd{d}{N}$ in \exref{example submanifold} is symmetric. 
  Therefore, by \thref{theorem pseudo-metric of dr space} and \remref{remark symmetric metric}, there exists a metric on $\embmfd{d}{N}$ which induces both the uniformity and the topology on $\embmfd{d}{N}$. 
  This provides a simpler way to recover the result of \metricemb, Theorem 4.4.
\end{ex}

\section{Diffeology defined by \texorpdfstring{$d_r$}{dr}}

In this section, we introduce a diffeology defined by the family $\{d_r\}_{r>0}$. 
We begin by defining \textit{diffeological spaces} and see their basic properties. 
All definitions and properties presented here are based on \IZ~by Iglesias-Zemmour, the standard textbook of diffeology. 
For more details on diffeological spaces, see \IZ. 

\begin{defi}[\IZ, 1.5] \label{definition diffeology}
  Let $X$ be a set. A \textit{parametrization} of $X$ is a map $U\to X$ where $U$ is an open set of $\mathbb{R}^n$ for some $n\in\mathbb{Z}_{\geq 0}$.
  A \textit{diffeology} on $X$ is a set $\mathscr{D}_X$ of parametrizations (whose elements are called \textit{plots}) satisfying the following three axioms:
  \begin{description}
    \item[Covering] Every constant parametrization $U\to X$ is a plot.
    \item[Locality] Let $p\colon U\to X$ be a parametrization. If there exists an open covering $\{U_{\alpha}\}$ of $U$ such that ${p|}_{U_{\alpha}}\in \mathscr{D}_X$ for all $\alpha$, then $p$ itself is a plot.  
    \item[Smooth compatibility] For every plot $p\colon U\to X$, every open set $V$ in $\mathbb{R}^m$ and $f\colon V\to U$ that is smooth as a map between Euclidean spaces, $p\circ f\colon V\to X$ is also a plot.
  \end{description}
  A \textit{diffeological space} is a set equipped with a diffeology.
  We usually write the underlying set $X$ to represent a diffeological space $(X,\mathscr{D}_X)$.
\end{defi}

Throughout this paper, we refer to any open subset of $\mathbb{R}^n$ (for some $n$) as a \textit{Euclidean open set}.  
We also denote the domain of a plot $p$ by $U_p$ and call $p$ an \textit{$n$-plot} if its domain $U_p$ is an open subset of $\mathbb{R}^n$.

\begin{defi}[\IZ, 1.14] \label{definition smooth map}
  Let $X$ and $Y$ be diffeological spaces. 
  A map $f\colon X\to Y$ is said to be \textit{smooth} if for every plot $p\colon \plotdom{p}\to X$, the composition $f\circ p\colon \plotdom{p}\to Y$ is a plot of $Y$.
  The set of smooth maps from $X$ to $Y$ is denoted by $\smoothmap{X}{Y}$. 
  If $X$ is a Euclidean open set, then $C^\infty(X,Y)$ coincides with the set of all plots $X\to Y$ (\IZ, 1.16).
\end{defi}

Diffeological spaces and smooth maps between them form a category denoted by $\dflg$.
An isomorphism in this category is called a \textit{diffeomorphism}.  
The category $\dflg$ contains the category $\mfd$ of smooth manifolds and smooth maps as a full subcategory. 
In particular, $\dflg$ also contains the category $\encl$ of Euclidean open sets and smooth maps as a full subcategory. This follows from the next example. 

\begin{ex} \label{example definition manifold diffeology}
  Let $M$ be a smooth manifold, and let $\mathscr{D}_M$ be the set of all smooth maps from some Euclidean open set $U$ to $M$ (in the sense of smooth maps between manifolds).
  Then $\mathscr{D}_M$ forms a diffeology on $M$. 
\end{ex}

\begin{defi}[\IZ, 2.8] \label{definition D-topology}
  Let $X$ be a diffeological space. The largest topology on $X$ that makes all plots continuous is called the \textit{$D$-topology} on $X$. 
  With this construction, we define the functor $D$ from the category $\dflg$ to the category of topological spaces $\topsp$.
\end{defi}

It is known that the functor $D$ has the right adjoint $\conti$ (Shimakawa-Yoshida-Haraguchi \SH, Proposition 2.1).
For a manifold $M$ equipped with the standard diffeology (\exref{example definition manifold diffeology}), the $D$-topology of $M$ coincides with the standard topology of the manifold $M$. 

\begin{defi}[\IZ, 1.66] \label{definition generating family}
    Let $X$ be a set, and let $\mathscr{F}$ be a family of parametrizations of $X$. Then, there exists the smallest diffeology $\mathscr{D}$ on $X$ that contains all parametrizations of $\mathscr{F}$ as plots. 
    We call this diffeology $\mathscr{D}$ the \textit{diffeology generated by} $\mathscr{F}$. 
    A parametrization $p\colon \plotdom{p}\to X$ is a plot of $\mathscr{D}$ if and only if it satisfies the following condition:
    there exists an open covering $\{U_{\alpha}\}$ of $\plotdom{p}$ such that for each $\alpha$, either ${p|}_{U_{\alpha}}$ is a constant parametrization 
    or ${p|}_{U_{\alpha}}=q\circ f$ for some $q\in\mathscr{F}$ and some smooth map $f\colon U_{\alpha} \to \plotdom{q}$ (\IZ, 1.68).
  
    If the family $\mathscr{F}$ satisfies $\bigcup_{q\in \mathscr{F}} q(\plotdom{q})=X$, then $\mathscr{F}$ is called a \textit{covering generating family}. 
    In this case, a parametrization $p\colon \plotdom{p}\to X$ is a plot of $\mathscr{D}$ if and only if it satisfies the following property:
    there exists an open covering $\{U_{\alpha}\}$ of $\plotdom{p}$ such that for each $\alpha$, ${p|}_{U_{\alpha}}=q\circ f$ for some $q\in\mathscr{F}$ and some smooth map $f\colon U_{\alpha} \to \plotdom{q}$.
\end{defi}

We now proceed to define a diffeology on the space with the family $\{d_r\}_{r>0}$. 

\begin{prop} \label{proposition definition diffeology on dr space}
    Let $(X, \{d_r\}_{r>0})$ satisfy the weaker $\{d_r\}$-axioms in \defref{definition weaker triangle inequality}.
    A parametrization $p\colon\plotdom{p}\to X$ is \textit{smooth at} $t_0\in \plotdom{p}$ if there exists $\delta>0$ satisfying the following conditions: 
    \begin{itemize}
        \item The map $\openball{\delta}{t_0}\to \mathbb{R}; t\mapsto d_r(p(t_0),p(t))$ is continuous for each $r>0$. 
        \item For any $t_1\in \openball{\delta}{t_0}\setminus p^{-1}(p(t_0))$, the map $\openball{\delta}{t_0}\to \mathbb{R}; t\mapsto d_r(p(t_1),p(t))$ is smooth at $t_0$ for each $r>0$. 
    \end{itemize}
    A parametrization $p\colon\plotdom{p}\to X$ is said to be a \textit{plot} if it is smooth at every point $t_0\in \plotdom{p}$. 
    Then, the set of all plots of $X$ forms a diffeology on $X$. 
\end{prop}

\begin{proof}
Since the covering and locality axioms follow directly from the definitions, we proceed to prove the smooth compatibility axiom. 
Let $p\colon\plotdom{p}\to X$ be a plot of $X$, and let $f\colon V\to \plotdom{p}$ be a smooth map from a Euclidean open set. 
Let $t_0\in V$ be arbitrary. 
Since $p$ is smooth at $f(t_0)$, there exists $\delta>0$ such that the following conditions hold: 
\begin{itemize}
    \item The map $\openball{\delta}{f(t_0)}\to \mathbb{R}; t\mapsto d_r(p\circ f(t_0),p(t))$ is continuous for each $r>0$. 
    \item For any $t_1\in \openball{\delta}{f(t_0)}\setminus p^{-1}(p\circ f(t_0))$, the map $\openball{\delta}{f(t_0)}\to \mathbb{R}; t\mapsto d_r(p\circ f(t_1),p(t))$ is smooth at $f(t_0)$ for each $r>0$. 
\end{itemize} 
By the continuity of $f$, there exists $\delta'>0$ such that $\openball{\delta'}{t_0}\subset f^{-1}(\openball{\delta}{f(t_0)})$. 
Using the identity 
\[
    [t\mapsto d_r(p\circ f(s),p\circ f(t))]=[t\mapsto d_r(p\circ f(s),p(t))]\circ f
\]
for any $s\in \openball{\delta'}{t_0}$, we obtain:
\begin{itemize}
    \item The map $\openball{\delta'}{t_0}\to \mathbb{R}; t\mapsto d_r(p\circ f(t_0),p\circ f(t))$ is continuous. 
    \item For any $t_1\in \openball{\delta'}{t_0}\setminus (p\circ f)^{-1}(p\circ f(t_0))$, the map $\openball{\delta'}{t_0}\to \mathbb{R}; t\mapsto d_r(p\circ f(t_1),p\circ f(t))$ is smooth at $t_0$ for each $r>0$. 
\end{itemize}
These conditions imply that $p\circ f$ is smooth at $t_0$, and therefore, $p\circ f$ is a plot of $X$. 
\end{proof}

In \tilingdflg, a diffeology is introduced on the tiling space. 
We aim to show that this diffeology coincides with the one defined using the family $\{d_r\}_{r>0}$ (\proref{proposition definition diffeology on dr space}). 
First, we review the diffeology on the tiling space defined in \tilingdflg. 

\begin{defi}[\tilingdflg, Definition 12] \label{definition tiling diffeology}
    Let $T$ be a tiling of $\mathbb{R}^d$. For each $T_1\in\tilingsp{T}$, we define the map 
    \[
        F_{T_1}\colon \mathbb{R}^d\to \tilingsp{T};x\mapsto T_1+x.
    \]
    The diffeology on $\tilingsp{T}$ defined in \tilingdflg~is the diffeology generated by the family $\{F_{T_1}\mid T_1\in \tilingsp{T}\}$. 
\end{defi}

We now prove that this diffeology coincides with the one defined using the family $\{d_r\}_{r>0}$ on $\tilingsp{T}$ (\proref{proposition definition diffeology on dr space}). 

\begin{thm} \label{theorem tiling diffeology is dr diffeology}
    Let $T$ be a tiling of $\mathbb{R}^d$, and let $\{d_r\}_{r>0}$ be the family introduced in \proref{proposition definition tiling dr}. 
    The diffeology on $\tilingsp{T}$, defined in \defref{definition tiling diffeology}, coincides with the one induced by the family $\{d_r\}_{r>0}$ (see \proref{proposition definition diffeology on dr space}).
\end{thm}


\begin{proof}
Let $\prototile$ be the set of prototiles of $T$. For each $\tau\in \prototile$, choose $\lambda_\tau>0$ such that
\[
    \Int(\tau)\cap \Int(\tau+v)\neq\emptyset
\]
for every $v\in\mathbb{R}^d$ with $0<|v|<2\lambda_\tau$, and set
\[
    \lambda_T:=\min_{\tau\in \prototile}\lambda_\tau>0.
\]
Then, for every $S\in\Omega_T$, every nonempty subset $K\subset \mathbb{R}^d$, and any $u,v\in\mathbb{R}^d$ with $|u|,|v|<\lambda_T$, if
$(S+u)\cap K=(S+v)\cap K$, then $u=v$.

First, we prove that $F_{T_1}$ is a plot of the diffeology induced by $\{d_r\}_{r>0}$ for each $T_1\in\tilingsp{T}$.
Let $T_1\in \tilingsp{T}$ be a tiling, and let $t_0\in \mathbb{R}^d$ be arbitrary. Set $\delta=\lambda_T/2$. Then, for each $t\in \openball{\delta}{t_0}$, we have
\[
    d_r(F_{T_1}(t_0),F_{T_1}(t))=d_r(T_1+t_0,T_1+t)=|t_0-t|.
\]
Indeed, since $(T_1+t_0)+(t-t_0)=T_1+t$, we have $d_r(T_1+t_0,T_1+t)\le |t_0-t|$. 
If strict inequality held, then there would exist $v\in\mathbb{R}^d$ with $|v|<|t_0-t|<\lambda_T$ such that $((T_1+t_0)+v)\cap rD=(T_1+t)\cap rD$. 
Since also $(T_1+t_0)+(t-t_0)=T_1+t$ and $|t-t_0|<\lambda_T$, the defining 
property of $\lambda_T$ implies $v=t-t_0$, a contradiction.
Moreover, for any $t_1\in\openball{\delta}{t_0}$, the same argument shows that
$d_r(F_{T_1}(t_1),F_{T_1}(t))=|t_1-t|$ for each $t\in\openball{\delta}{t_0}$,
since $t,t_1\in\openball{\delta}{t_0}$ implies $|t_1-t|<2\delta=\lambda_T$.
In particular, the map $t\mapsto d_r(F_{T_1}(t_1),F_{T_1}(t))=|t_1-t|$
is smooth at $t_0$ for each $t_1\neq t_0$.
Thus, $F_{T_1}$ satisfies the conditions specified in 
\proref{proposition definition diffeology on dr space} for each $T_1\in\tilingsp{T}$.


Next, we prove the other inclusion. 
Let $p\colon\plotdom{p}\to \tilingsp{T}$ be a plot in the sense of \proref{proposition definition diffeology on dr space}, and let $t_0\in \plotdom{p}$. 
Then, there exists $\delta>0$ such that the map 
\[
    \openball{\delta}{t_0}\to \mathbb{R};\ t\mapsto d_r(p(t_0),p(t))
\]
is continuous for each $r>0$. 
Choose $\varepsilon>0$ such that $0<\varepsilon<\lambda_T/2$.
Let $r>0$ be arbitrary. By the continuity of the map $t\mapsto d_r(p(t_0),p(t))$, we can choose $0<\delta'\leq \delta$ such that $d_r(p(t_0),p(t))<\varepsilon$ for all $t\in \openball{\delta'}{t_0}$. 
We claim that this $\delta'$ satisfies $d_{r'}(p(t_0),p(t))<\varepsilon$ for all $r'>0$ and all $t\in \openball{\delta'}{t_0}$. 
To prove this, define
\[
    A_{r'}=\{t\in\openball{\delta'}{t_0}\mid d_{r'}(p(t_0), p(t))<\varepsilon\}.
\]
We only need to consider the case where $r'>r$, as $r'\leq r$ implies $A_{r'}\supset A_r=\openball{\delta'}{t_0}$. 
Since the map $t\mapsto d_{r'}(p(t_0),p(t))$ is continuous on $\openball{\delta}{t_0}$, $A_{r'}\subset \openball{\delta'}{t_0}$ is open. 
Since $\openball{\delta'}{t_0}$ is connected and $t_0\in A_{r'}$, it suffices to show that $A_{r'}$ is closed in $\openball{\delta'}{t_0}$. 
To prove this, suppose there exists $t\in \overline{A_{r'}}\setminus A_{r'}$. 
Since $t\in \openball{\delta'}{t_0}$, there exists $v\in \mathbb{R}^d$ such that $|v|<\varepsilon$ and 
\[
    (p(t_0)+v)\cap \closedball{r}=p(t)\cap \closedball{r}. 
\]
Since $t\notin A_{r'}$, the same $v$ does not satisfy
\[
    (p(t_0)+v)\cap \closedball{r'}=p(t)\cap \closedball{r'}.
\]
If $d_{r'}(p(t_0),p(t))<\lambda_T$, then there exists $u\in\mathbb{R}^d$ with $|u|<\lambda_T$ such that $(p(t_0)+u)\cap \closedball{r'}=p(t)\cap \closedball{r'}$. Restricting to $\closedball{r}$, we also have $(p(t_0)+u)\cap \closedball{r}=p(t)\cap \closedball{r}$, and since $|v|<\varepsilon<\lambda_T$, the defining property of $\lambda_T$ implies $u=v$, a contradiction. Hence,
\[
    d_{r'}(p(t_0),p(t))\geq \lambda_T>\varepsilon.
\]
On the other hand, since $t\in \overline{A}_{r'}$, there exists a sequence $\{s_n\}_{n\in \mathbb{Z}_{\geq 0}}$ converging to $t$ such that $d_{r'}(p(t_0),p(s_n))<\varepsilon$ for all $n$. 
This contradicts the continuity of the map $t\mapsto d_{r'}(p(t_0), p(t))$ at the point $t$. 
Thus, $A_{r'}$ is closed in $\openball{\delta'}{t_0}$. Therefore, we obtain $A_{r'}=\openball{\delta'}{t_0}$.

By the above argument, for each $t\in \openball{\delta'}{t_0}$ and each $r'>0$, there exists $u\in\mathbb{R}^d$ such that $|u|<\varepsilon$ and
\[
    (p(t_0)+u)\cap \closedball{r'}=p(t)\cap \closedball{r'}.
\]
By the defining property of $\lambda_T$, such a vector $u$ is unique, and hence independent of $r'$. Therefore, for each $t\in\openball{\delta'}{t_0}$, there exists a unique $u_t\in\mathbb{R}^d$ such that $|u_t|<\varepsilon$ and
\[
    (p(t_0)+u_t)\cap \closedball{r'}=p(t)\cap \closedball{r'}
\]
for every $r'>0$. In particular, $p(t_0)+u_t=p(t)$ for each $t\in\openball{\delta'}{t_0}$. Thus, we obtain a unique map $u\colon \openball{\delta'}{t_0}\to \mathbb{R}^d;\ t\mapsto u_t$ such that $p(t_0)+u(t)=p(t)$ for each $t\in\openball{\delta'}{t_0}$.
We now show that this map $u$ is smooth. 

First, we prove the continuity of $u$ at the point $t_0\in\openball{\delta'}{t_0}$. 
It has already been proved that the map $d_r(p(t_0), p(\cdot))\colon\openball{\delta'}{t_0}\to \mathbb{R}$ is continuous. 
Since $d_r(p(t_0),p(t))=|u(t_0)-u(t)|$ for each $t\in\openball{\delta'}{t_0}$, for any $\varepsilon'>0$, there exists $\eta>0$ such that for each $t\in\openball{\eta}{t_0}$,
\[
    |u(t)-u(t_0)|=\bigl|d_r(p(t_0),p(t))-d_r(p(t_0),p(t_0))\bigr|<\varepsilon'.
\]
Hence, $u$ is continuous at $t_0$.



Next, we prove the smoothness of $u$ at the point $t_0\in\openball{\delta'}{t_0}$.
If $p(t)=p(t_0)$ holds for all $t\in\openball{\delta'}{t_0}$, then $u$ is constant on $\openball{\delta'}{t_0}$, and in particular, smooth on $\openball{\delta'}{t_0}$.
Otherwise, there exists $t_1\in\openball{\delta'}{t_0}\setminus p^{-1}(p(t_0))$.
By repeating this process, we can find a finite sequence $t_1, \dots, t_n\in \openball{\delta'}{t_0}$ such that
\[
    \{u(t)\mid t\in\openball{\delta'}{t_0}\}\subset \Span\{\overrightarrow{u(t_0)u(t_i)}\mid i=1,\dots,n\}
\]
and the set of vectors $\left\{\overrightarrow{u(t_0)u(t_i)}\right\}_{i=1}^n$ is linearly independent.
Thus, we can restrict the codomain of $u$ to $\Span\left\{\overrightarrow{u(t_0)u(t_i)}\mid i=1,\dots,n\right\}\cong \mathbb{R}^n$.
Note that $|u_t| < \varepsilon < \lambda_T/2$ for each $t \in \openball{\delta'}{t_0}$,
so for any $t, t_i \in \openball{\delta'}{t_0}$, we have
$|u(t)-u(t_i)| \leq |u(t)|+|u(t_i)| < 2\varepsilon < \lambda_T$.
Therefore, the same uniqueness argument gives
$d_r(p(t_i),p(t))=|u(t_i)-u(t)|$
for all $t,t_i\in\openball{\delta'}{t_0}$.
We define the map
\[
    f\colon \openball{\delta'}{t_0}\times \mathbb{R}^n\to \mathbb{R}^n;(t,x)\mapsto (|x-u(t_i)|^2-d_r(p(t_i), p(t))^2)_{i=1}^n.
\]
In particular, $f(t_0, u(t_0)) = 0$ and $f(t, u(t)) = 0$ for all $t\in\openball{\delta'}{t_0}$.
Since the map $d_r(p(t_i), p(\cdot))$ is smooth at $t_0$ for each $t_i$, the map $f$ is smooth around $(t_0, u(t_0))$.
By the linear independence of $\left\{\overrightarrow{u(t_0)u(t_i)}\right\}_{i=1}^n$, we have
\[
    \left.\det\left({\frac{\partial f}{\partial x}}(t_0, u(t_0))\right)=\det{\begin{pmatrix}
        {}^t (2x-2u(t_1)) \\ \vdots \\ {}^t(2x-2u(t_n))
       \end{pmatrix}}\right|_{t=t_0, x=u(t_0)}\neq 0.
\]
Thus, by the implicit function theorem, there exist open sets $U\subset \openball{\delta'}{t_0}$ and $V\subset \mathbb{R}^n$ with $t_0\in U$ and $u(t_0)\in V$, and a smooth map $g\colon U\to V$ such that
for all $(t,x)\in U\times V$, 
\[
    f(t,x)=0\iff x=g(t). 
\]
This map $g$ is precisely the restriction ${u|}_{U}$, which implies that $u$ is smooth at $t_0$. 
\end{proof}

\appendix

\section{Relating topologies and diffeologies}

We have seen that the topology and diffeology of the tiling space can be connected through the structure of the family $\{d_r\}_{r>0}$.


It is natural to ask whether a similar relationship holds for the space of submanifolds $\embmfd{d}{N}$, equipped with the family introduced in \exref{example submanifold}.

In \mfdembdflg, a diffeology on $\embmfd{d}{N}$ is defined as follows:  
a parametrization $p\colon\plotdom{p}\to \embmfd{d}{N}$ is a plot if  
\[
\Gamma_p = \{(t,x)\in \plotdom{p}\times \mathbb{R}^N \mid x \in p(t)\}
\]
is a closed submanifold of $\plotdom{p}\times\mathbb{R}^N$, and the projection  
\[
\hat{p} \colon \Gamma_p \to \plotdom{p}, \quad (t,x) \mapsto t
\]
is a submersion.

It is natural to ask how this diffeology relates to the one induced by $\{d_r\}_{r>0}$ (see \proref{proposition definition diffeology on dr space}). 
At least, the following fact holds in general. 

\begin{lemma} \label{lem smooth plot is continuous}
    Let $(X, \{d_r\}_{r>0})$ satisfy the weaker $\{d_r\}$-axioms in \defref{definition weaker triangle inequality}. 
    If a map $p\colon\plotdom{p}\to X$ is a plot of $X$ as defined in \proref{proposition definition diffeology on dr space}, then $p$ is also continuous. 
\end{lemma}

\begin{proof}
Let $t_0\in \plotdom{p}$, $r>0$, and $\varepsilon>0$ be arbitrary. 
Since $p$ is a plot of $X$, there exists $\delta>0$ such that $d_r(p(t_0), p(t))<\varepsilon$ for all $t\in\openball{\delta}{t_0}$. 
This implies that $\openball{\delta}{t_0}\subset p^{-1}(\entourage{r}{\varepsilon}[p(t_0)])$. 
\end{proof}

\begin{prop} \label{proposition D-topology is finer than topology}
    Let $(X, \{d_r\}_{r>0})$ satisfy the weaker $\{d_r\}$-axioms in \defref{definition weaker triangle inequality}. 
    The $D$-topology of the diffeology on $X$ induced by the family $\{d_r\}_{r>0}$ (see \proref{proposition definition diffeology on dr space}) is finer than 
    the topology on $X$ induced by the same family (see \defref{definition neighborhoods of topology}).
\end{prop}

\begin{proof}
Let $x\in X$, $r>0$, and $\varepsilon>0$ be arbitrary. For each plot $p\colon \plotdom{p}\to X$ of $X$, 
the preimage $p^{-1}(\nhddr{r}{\varepsilon}{x})$ is open in $\plotdom{p}$ because $p$ is continuous (\lemref{lem smooth plot is continuous}). 
Therefore, $\nhddr{r}{\varepsilon}{x}$ is $D$-open. 
\end{proof}

\section*{Acknowledgement}




I would like to express my sincere gratitude to my supervisor, Takuya Sakasai for his valuable guidance and continuous support throughout this research. 
I am also grateful to Katsuhiko Kuribayashi for many insightful discussions and helpful advice. 
I thank Toshiyuki Kobayashi for his generous support and encouragement as my supporting supervisor in the WINGS-FMSP program.  
I am grateful to Ryotaro Kosuge for helpful conversations. 
This work was supported by JSPS Research Fellowships for Young Scientists and KAKENHI Grant Number JP24KJ0881. 
Lastly, I would like to acknowledge the WINGS-FMSP program for its financial support. 

\bibliographystyle{plain}
\bibliography{ParametrizedMetric}


\end{document}